\documentclass[a4paper,10pt]{article}

\usepackage{amsmath,amssymb,amsthm,
verbatim,
}

\usepackage{tikz}
\usetikzlibrary{shapes.geometric,calc,decorations.markings,math}
\usetikzlibrary{trees,snakes,decorations.pathreplacing}

\tikzstyle{nodino}=[circle,draw,fill,inner sep=0pt]

\tikzstyle{nodo}=[circle,draw,fill,inner sep=0pt,minimum size=%
1.5mm]

\tikzstyle{infinito}=[circle,inner sep=0pt,minimum size=0mm]

 \newcommand{\eps}{{\varepsilon}}

\newcommand\R{{\mathbb R}}

\newcommand\Hmu{{H_\mu^1}}
\newcommand\Hr{{H_{\mu,r}^1}}
\newcommand\T{\mathcal T}
\newcommand\To{{\mathcal T}_o}
\newcommand\f{\frac}
\newcommand\dx{{\,dx}}

\newcommand\elevel{{\mathcal L}}

\newcommand\eign{\lambda_1}

\newcommand\musTor{\mu^*_{\To,r}}
\newcommand\musTr{\mu^*_{\T,r}}
\newcommand\musTo{\mu^*_{\To}}
\newcommand\musT{\mu^*_{\T}}
\newcommand\musG{\mu^*_{\G}}
\newcommand\musGr{\mu^*_{\G,r}}

\newcommand\G{\mathcal G}
\newcommand\B{\mathcal B}

\newcommand\C{\mathcal C}

\newcommand\be{\begin{equation}}
\newcommand\ee{\end{equation}}

\newcommand\radice{o}

\newtheorem{theorem}{Theorem}[section]
\newtheorem{proposition}[theorem]{Proposition}
\newtheorem{lemma}[theorem]{Lemma}
\newtheorem{corollary}[theorem]{Corollary}

\theoremstyle{remark}
\newtheorem{remark}[theorem]{Remark}
\newtheorem*{remark*}{Remark}

\theoremstyle{definition}

\date{}

\title{NLS ground states on metric trees: existence results and open questions}

\author{Simone Dovetta\thanks{We acknowledge that the present research has been partially supported by MIUR grant Dipartimenti di Eccellenza 2018-2022 (E11G18000350001)}, Enrico Serra,
Paolo Tilli \\ \ \\{\small  Dipartimento di Scienze
Matematiche ``G.L. Lagrange'', Politecnico di Torino } \\ {\small
Corso Duca degli Abruzzi, 24, 10129 Torino, Italy}}

\begin{document}

\maketitle

\begin{abstract} We consider the minimization of the NLS energy on a metric tree, either rooted or unrooted, subject to a mass constraint. With respect to the same problem on other types of metric graphs, several new features appear, such as the existence of minimizers with positive energy, and the emergence of unexpected threshold phenomena.  We also study the problem with a radial symmetry constraint that is in principle different from the free problem due to the failure of the P\'olya--Szeg\H{o} inequality for radial rearrangements. A key role is played by a new Poincar\'e inequality with remainder.

\end{abstract}

\smallskip


\section{Introduction}
In recent years, a large and still increasing interest has been devoted to the investigation of nonlinear dynamics on metric graphs or networks. Conceiving graphs as a meaningful model of ramified structures, and driven by physical applications, thorough investigations have been carried out first for {\em nonlinear Schr\"odinger equations} (NLS) (see for instance \cite{ACFN1,ACFN2,PS} and the review \cite{N}), and more recently also for {nonlinear Dirac equations} (see \cite{BCT,BCT1}).

Particularly, efforts have been focused on the analysis of {\em standing waves}, i.e. solutions of the corresponding stationary equations. Within this framework, there has been an intensive study of the existence
of mass-constrained ground states for the NLS energy, that is, global minimizers of the energy among functions of prescribed $L^2$ norm. This problem has been initially considered in the case of graphs made of a core of finitely
many bounded edges, and a finite number of unbounded edges (halflines)
attached to it, and this setting is nowadays quite well understood (we refer to \cite{AST-CV,AST-JFA,AST-CMP} for the nonlinearity extended to the whole graph, and to \cite{DT-p,DT, ST-JDE,ST-NA,T-JMAA} for the nonlinearity concentrated on the sole compact core). Similar results have then been accomplished also in the case of compact graphs (\cite{CDS,D-jde,MP-amre16}).

More recently,
however, another interesting class of graphs has been investigated,
where
noncompactness is no longer due to the presence of unbounded edges, but rather
to the infinite number of bounded edges, arranged to form an infinite periodic structure (\cite{AD,ADST,D-nodea,Pa}). A prototypical case study can
be found in \cite{ADST}, where the graph is
a planar grid with vertices on the lattice ${\mathbb Z}^2$, connected
by vertical and horizontal edges of length one: the main interesting
feature is that, even though such an ambient space is of course one-dimensional
(at least locally), at large scales the overall behaviour turns out to be
two-dimensional, to the extent that a Sobolev inequality holds true, formally identical
to the Sobolev inequality in $\R^2$. This makes the functional analysis quite rich
and interesting if compared to classical graphs with halflines, with several unexpected
consequences on the ground-state problem, such as a continuum of critical exponents.

\begin{figure}
\begin{center}
\begin{tikzpicture}[xscale= 1,yscale=1]
\begin{scope}[xscale= 0.35,yscale=0.35,grow=south,thick,
]
\tikzstyle{level 1}=[level distance=10em]
\tikzstyle{level 2}=[sibling distance=14em,level distance=7em]
\tikzstyle{level 3}=[sibling distance=8em,level distance=9.1em]
\tikzstyle{level 4}=[sibling distance=3.5em,level distance=5em]
\node at (0,12em) [nodo] {}   child {node [nodo] {}
node [nodo] {} child foreach \x in {0,1}
{node [nodo] {} child foreach \y in {0,1}
{node [nodo] {} child [dashed] foreach \z in {0,1}}}};
\end{scope}
\begin{scope}[xshift=17em,grow cyclic,shape=circle, thick,
level distance=2.7em,
                  cap=round]
\tikzstyle{level 1}=[rotate=-90,sibling angle=120]
\tikzstyle{level 2}=[sibling angle=85]
\tikzstyle{level 3}=[sibling angle=57]
\tikzstyle{level 4}=[sibling angle=60]
\tikzstyle{edge from parent}=[draw]
\node at (0,0) [nodo] {} child  foreach \A in {1,2,3}
    { node [nodo]
       {} child  foreach \B in {1,2}
        { node [nodo] {} child   foreach \C in {1,2}
           { node [nodo] {} child [dashed,level distance=1.5em]  foreach \C in {1,2} }
        }
    };
\node at (-20em,7em) {(a)};
\node at (-10em,7em) {(b)};
\end{scope}
\end{tikzpicture}
\end{center}
\caption{(a) The rooted tree $\To$. (b) The unrooted tree $\T$.}\label{Figalberi}
\end{figure}
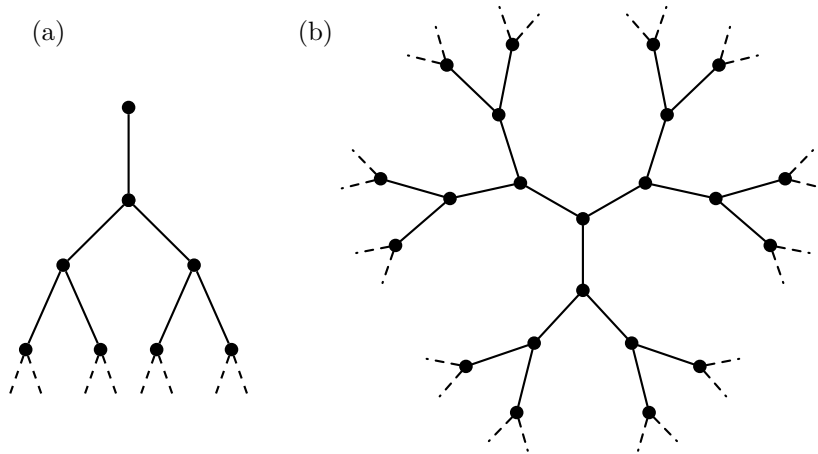

In this paper we consider the case where the graph is a \emph{binary tree} (either rooted or unrooted)
made up of infinitely many edges of length one, so that every vertex has degree
three (except of course for the root which, if present, has degree one).
These two graphs, denoted respectively by $\To$ and $\T$, are depicted in
Figure~\ref{Figalberi}.

Among various topics on quantum graphs, metric trees have been gathering a significant interest since the early years. Specifically, many efforts have been devoted to the analysis of differential operators on such graphs, with a prominent focus on spectral properties. 
The study of Hardy-type integral operators on trees started in \cite{evans0,evans} within the investigation of the Neumann Laplacian on irregular domains, whereas \cite{carlson} was the first paper to unravel the gap-structure of the Neumann spectrum on homogeneous trees. Later, several papers studied the spectrum of Schr\"odinger and Laplacian operators on trees (\cite{carlson1,naymark,sobolev,solomyak} and references therein).

Although
the
Laplacian on such graphs
has been intensively studied, the problem of
ground states
for the NLS energy
\begin{equation}
\label{nlse}
E(u,\G):=\frac 1 2 \int_\G|u'|^2\dx-\f1p\int_\G |u|^p\dx
\end{equation}
(where $\G=\To$ or $\G=\T$) has not yet been investigated (up to our knowledge, the only discussion of nonlinear issues on trees is the one in \cite{BC}, which is not related to the problem we consider here). In fact, as we will see, several interesting phenomena
arise, together with some basic questions that remain open and call for further investigation.

Given a metric graph $\G$, a number $p\in (2,6)$ and
a mass $\mu>0$, the ``ground state problem'' on $\G$ is the minimization problem
\begin{equation}
\label{GSP}
\min_{u\in \Hmu(\G)} E(u,\G),
\end{equation}
where $E(u,\G)$ is as in \eqref{nlse} and
$\Hmu(\G)$ is the class of mass-constrained functions
	\begin{equation}\label{defHumu}
	\Hmu(\G):=\left\{u\in\ H^1(\G)\,:\,\int_\G
	|u|^2\dx=\mu\right\}.
	\end{equation}
Any solution to this problem, i.e. any function $u\in \Hmu(\G)$ achieving the
ground-state energy level
	\begin{equation}
	\label{elevel}
	\elevel_\G(\mu):=\inf_{u\in \Hmu(\G)}E(u,\G),\quad \mu\geq 0,
	\end{equation}
is called a \emph{ground state} of mass $\mu$: throughout this paper, we shall be
concerned with problem \eqref{GSP} and related questions, when $\G$ is either $\To$ or $\T$.

Contrary to graphs with halflines and to periodic graphs, where the ground-state energy defined in \eqref{elevel}
is always \emph{nonpositive} (and typically negative in
mass regimes where ground states exist), when $\G=\To$ (or $\T$)
$\elevel_\G(\mu)$ is \emph{strictly positive} if $\mu$ is small enough and,
more interestingly, there are always
mass regimes $\mu$ where ground states exist even though $\elevel_\G(\mu)>0$.
This is so because, as is well known, $\G$  supports
a Poincar\'e inequality
\begin{equation}
\label{poinc1}
\int_{\G} |u'(x)|^2\,dx \geq \lambda_1 \int_{\G} |u(x)|^2\,dx\quad\forall u\in H^1(\G),
\end{equation}
where the best constant
\begin{equation}
\label{rayleigh}
\eign:=\inf_{\substack{u\in H^1(\T)\\
u\not\equiv 0}}\frac{\int_\T|u'|^2\dx}{\int_\T|u|^2\dx}=
\inf_{\substack{u\in H^1(\To)\\
u\not\equiv 0}}\f{\int_{\To}|u'|^2\dx}{\int_{\To}|u|^2\dx}>0
\end{equation}
is the same for both $\T$ and $\To$, as it is very simple to see. Thus, if $u$ has mass $\mu$,
the kinetic term in \eqref{nlse}
is not smaller than
$\lambda_1\mu/2$
and,
for small $\mu$, this positive lower bound prevails over
the second integral and governs the behaviour of $\elevel_\G(\mu)$. More precisely, we have the following result.
\begin{theorem}[energy-level function]
		\label{THM0}
		Let $\G=\T$ (or $\G=\To$) and $p\in (2,6)$. The function $\elevel_\G(\mu)$
in \eqref{elevel}  is concave,  satisfies $\elevel_\G(0)=0$,
\begin{equation}
\label{sottoretta}
\elevel_\G (\mu) \le \frac12 \lambda_1\mu\qquad \forall \mu \ge 0,
\end{equation}	
and is differentiable 
at $\mu=0^+$ with
\begin{equation}\label{derivL}
\elevel'_\G(0^+)=\frac {\lambda_1}2>0,
\end{equation}
so that $\elevel_\G(\mu)$ is strictly positive and increasing if $\mu$
is small enough. Moreover, as $\mu$ increases, $\elevel_\G(\mu)$ achieves
an absolute maximum, after which it becomes decreasing and eventually
negative.
\end{theorem}
Furthermore, at least when $p\in [4,6)$, not only does $\elevel_\G(\mu)$ detach
from zero with a slope equal to $\lambda_1/2$, it is in fact a \emph{linear function}
of $\mu$ up to a certain threshold $\musG$, which is also the precise
mass threshold beyond which ground states exist:
	\begin{theorem}
		\label{THM 1}
		Assume $\G=\T$ (or $\G=\To$) and $p\in (2,6)$, and define
\begin{equation}
\label{critmass}
\musG := \max\left\{\mu \ge 0\,:\,\elevel_\G(\mu)=\f12\eign\mu\right\},
\end{equation}
so that
		\begin{equation}
			\label{inf}
			\elevel_\G(\mu)\,\,\,\begin{cases}
				\,\,=\displaystyle \f12\eign\mu & \quad\text{if $\mu\in [0,\musG]$}\\{}\\
				\,\,<\displaystyle \f12\eign\mu & \quad\text{if $\mu>\musG$.}
			\end{cases}
		\end{equation}
Then
ground states of mass $\mu$ exist if $\mu>\musG$, whereas they do not exist if $\mu\in (0,\musG)$.
Moreover:
		\begin{itemize}
			\item[(i)] if $p\in[4,6)$ then $\musG>0$, so that at $\mu=\musG$
there is a genuine transition from nonexistence to existence of ground states;
			\item[(ii)] if $p\in (4,6)$, then ground states exist also
when $\mu=\musG$.
		\end{itemize}
	\end{theorem}

\begin{remark}\label{remopen}
  When $p\in (2,4)$ we are not able to prove that $\musG>0$: if, as we believe,
  this was the case, then at $\mu=\musG$
  a transition would occur from nonexistence to existence of ground states (exactly
  as when $p\in [4,6)$), while, if $\musG=0$, then
  ground states would exist for every mass.
 Finally, when $p=4$, we do not know whether ground states exist or not, with a mass $\mu=\musG$.
\end{remark}

Another interesting issue concerns \emph{symmetries}.
A function $u\in H^1(\To)$ is called {\em radial} if its value at any point
$x\in\To$
depends only on $|x|$, the geodesic distance of $x$ from the root $\radice$ of the tree
(the same definition applies to functions $u\in H^1(\T)$,  choosing as  $\radice$
an arbitrary vertex of $\T$).
Unfortunately,   the usual  techniques of radial symmetrization (see \cite{AST-CV})
are not effective in this setting,
and it remains an \emph{open problem}
to establish  whether ground states (on $\To$ or on $\T$)
are necessarily radial functions.
Indeed,  given a
nonnegative function $u\in H^1(\To)$, one
may consider its \emph{radial rearrangement} $u^*\in H^1(\To)$, i.e.
the unique radial function $u^*$ which is
equimeasurable with $u$ (see e.g. \cite{AST-CV}). Since the passage from $u$ to $u^*$ preserves
  the mass and the $L^p$ norm (in fact, every $L^r$ norm),
the radiality of ground states would follow,
 if only one could rely on the usual P\'olya--Szeg\H{o}
  inequality
\begin{equation}
  \label{PSineq}
  \int_{\To} \left|(u^*)'(x)\right|^2\,dx\leq\int_{\To} |u'(x)|^2\,dx\quad\forall u\in H^1(\To).
\end{equation}
Unfortunately, however, this
standard construction of radial competitors  is not useful
because, as
simple examples
show, \eqref{PSineq} is not valid in general (the main reason is
that metric balls in $\To$ fail to be isoperimetric sets,
but we shall not pursue this issue here).

In the light of these facts, it is natural to define
\begin{equation}\label{defHunomur}
\Hr(\G):=\{\,u\in\Hmu(\G)\,:\,\text{$u$ is radial}\,\}
\end{equation}
where $\G=\To$ or $\G=\T$,
and study the minimization problem
\begin{equation}
\label{GSPrad}
\min_{u\in \Hr(\G)} E(u,\G),
\end{equation}
i.e. the ground state problem \emph{restricted} to radial functions. Clearly, if
a function  $u$ that solves \eqref{GSP} was radial, then it would also solve
\eqref{GSPrad} (and if all ground states were radial functions, then the
two problems would be equivalent). Since, however, the radiality of ground states is
an open issue, the two problems should be considered as distinct, until proven otherwise.
Thus, also for the radial problem \eqref{GSPrad}
we will study the level function
	\begin{equation}
	\label{radlevel}
	\elevel_{\G,r}(\mu):=\inf_{u\in\Hr(\G)}E(u,\G),\quad\mu\geq 0,
	\end{equation}
and we will call  \emph{radial ground state} (of mass $\mu$)
any function $u\in \Hr(\G)$ that achieves the infimum in \eqref{radlevel}.

Our results for the radial problem are similar to, but actually more precise than,
the corresponding statements for the non-radial case. In particular,
in the radial case the mass threshold below which ground state do not exist (and
the level function is linear) is guaranteed to be strictly positive
for \emph{every} $p\in (2,6)$, not just for $p\in (4,6)$ as in Theorem~\ref{THM 1}.
	\begin{theorem}
		\label{THM 2}
Let $\G=\T$ (or $\G=\To$) and $p\in (2,6)$. Then, defining
\begin{equation}
\label{critmassrad}
\musGr := \max\left\{\mu \ge 0\,:\,\elevel_{\G,r}(\mu)=\f12\eign\mu\right\}
\end{equation}
there holds $\musGr>0$,
and the function $\elevel_{\G,r}(\mu)$ in \eqref{radlevel} is
concave and such that
\[
\elevel_{\G,r}(\mu)\,\,\,\begin{cases}
\,\,=\displaystyle \f12\eign\mu & \quad\text{if $\mu\in [0,\musGr]$}\\[1.1em]
\,\,<\displaystyle \f12\eign\mu & \quad\text{if $\mu>\musGr$}\\[1em]
\,\,<0 & \quad\text{if $\mu$ is large enough.}
\end{cases}
\]
Finally,
radial ground states of mass $\mu$ 
exist if $\mu \geq \musGr$, whereas they do not exist if $\mu\in(0,\musGr)$.
\end{theorem}
In the light of this, the problem raised in Remark \ref{remopen} would be
solved, if one could prove that ground states are in fact radial functions
or, at least, that $\elevel_\G(\mu)=\elevel_{\G,r}(\mu)$ (note that
the inequality 
$\leq$ is trivially satisfied).

Our proof techniques for the radial and the non-radial case are quite similar,
but in the radial case one can exploit the exponential decay of radial $H^1$
functions (see Lemma~\ref{lemmadecay}) to obtain stronger estimates.
In either  case, a central role is played by the well-known fact that the infimum in
\eqref{rayleigh} is not attained (\cite{kovarik}), so that if $u$ has mass $\mu$
then
\begin{equation}
  \label{rempos}
  R:=\f12\int_\G|u'|^2\dx-\frac 1 2\lambda_1\mu>0,
\end{equation}
combined with the
fact that ground states of mass $\mu$ exist, as soon as the energy in \eqref{nlse}
can be made smaller than $\lambda_1\mu/2$ (see Proposition \ref{levelargument}). And, clearly,
this is possible only if the last integral in \eqref{nlse} is bigger than
the ``remainder'' $R$ in \eqref{rempos}.
It is therefore crucial to establish sharper forms of the Poincar\'e inequality \eqref{poinc1},
to  somehow quantify the fact that (when $u\not\equiv 0$) the inequality is \emph{strict}.
In the framework of metric trees, such inequalities ``with a remainder term'' have
been intensively studied in \cite{kovarik, kovarik1}, and in particular it is known that
\begin{equation}
\label{rest_rad}
\int_\G|u'|^2\dx-\eign\int_\G|u|^2\dx\geq C\int_\G\frac{|u|^2}{(1+|x|)^2}\dx\quad
\forall u\in H^1(\G).
\end{equation}
Although the power decay of the weight $(1+|x|)^{-2}$ is optimal (see \cite{kovarik} Thm.~2.4),
this inequality is not sufficient for our purposes, thus we shall prove the following
new inequality (which might be of some interest in itself), with
the remainder estimated in terms of the $L^\infty$ norm.
	\begin{theorem}[Remainder term in the Poincar\'e inequality]\label{teopoinc}
	Assume $\G=\T$, or $\G=\To$. Then, for every $u\in H^1(\G)$,
	\begin{equation}
	\label{poincarest}
	\int_\G|u'|^2\dx-\eign\int_\G|u|^2\dx\geq C\|u\|_{L^\infty(\G)}^2
	\end{equation}	
	for some constant $C>0$ independent of $u$.
	\end{theorem}
From \eqref{poincarest} we see that a function $u$ in \eqref{nlse},
of a given mass $\mu$, may have a kinetic integral close to optimality
in \eqref{rayleigh}, only at the price of a small $L^\infty$ norm,
which reflects into a small $L^p$ norm in \eqref{nlse}.
Therefore, for a function in \eqref{nlse},
achieving a small kinetic term  and a large $L^p$ norm
are two conflicting goals. For quantitative reasons, when $\mu$ is very small,
it may be convenient to sacrifice the latter goal in favour of the former, and
this is the reason why ground states of small mass (e.g. when $p\geq 4$) may fail to exist.

As is well known, any ground state $u$ will satisfy the ODE
\begin{equation}
  \label{ode}
  u''+u |u|^{p-2} =\lambda u
\end{equation}
(where $\lambda$ is a Lagrange multiplier) on every edge of the tree $\G$, coupled with
the Kirchhoff condition
\begin{equation}
  \label{kir}
\sum_{e} \frac{du_e}{dx_e}(v)=0
\end{equation}
at every vertex of $\G$, where the sum is extended over all the edges $e$ incident at $v$
(see e.g. \cite{AST-CV}). If $\G=\To$ and $u\in H^1_{\mu,r}(\To)$ is
a \emph{radial} ground state, this equation can be easily visualized on $\R^+$,
as follows. Indeed,
in this case $u(x)=v(|x|)$ for a suitable continuous function $v$ defined on
$[0,+\infty)$,  so that \eqref{ode} translates to
\[
 v''(t)+v(t) |v(t)|^{p-2} =\lambda v(t)\quad \forall t\in (j,j+1),\quad
 j\in \{0,1,2,\ldots\},
\]
while \eqref{kir} turns into
\[
v'(0)=0,\quad
v'(j^+)=\frac 1 2 v'(j^-),\quad j\in \{1,2,3,\ldots\}.
\]
This rather simple visualization of radial ground states of course fails in the general nonradial setting
and this is one of the reasons why results in this case are harder to obtain. Our results are no exception, since, as we have already pointed out, Theorem \ref{THM 2} is more precise than Theorem \ref{THM 1}.

Finally we point out that all of our results hold without any modifications for homogeneous trees with vertices of arbitrary degree and edges of arbitrary length. 
\medskip

The paper is structured as follows. Section \ref{sec:prelim} contains some preliminary estimates and the proofs of Theorem \ref{THM0} and Theorem \ref{teopoinc}. In Section  \ref{sec:existence} we analyze the compactness issues and we establish the main results that lead to the proof of Theorem  \ref{THM 1}. Finally, in Section \ref{sec:rad} we discuss the radially symmetric case and prove Theorem  \ref{THM 2}.
\medskip

\noindent{\bf Notation}. We denote by $\|u\|_{L^p(\G)}$ the usual $L^p$ norm of a function $u$ defined on a metric graph $\G$. Whenever possible, and typically in the proofs, we also use the simplified notation $\|u\|_p$.

\section{Inequalities and a priori estimates}
	\label{sec:prelim}
This section is initially devoted to the proof of Theorem~\ref{teopoinc} and
some a priori estimates that will be used in the sequel. And lastly,
using these tools, we will prove Theorem~\ref{THM0}.

\proof[{\bf Proof of Theorem \ref{teopoinc}}]
		We first consider the case where $\G=\T$.
Given $u\in H^1(\T)$, up to considering $-u$ in place of $u$, we may
assume that $M:=\Vert u\Vert_\infty$ is the absolute maximum of $u$.
Let $x_0\in \T$ be a point where $u(x_0)=M$ and assume, for the moment,
that $x_0$ is not a vertex of $\T$ (see Figure \ref{figxz}.(a)).
Since $\T$ is a tree,  the removal of $x_0$ disconnects $\T$ and,
since $x_0$ is not a vertex,
$\T\setminus\{x_0\}$
consists of two connected components whose closures (relative to $\T$)
will be denoted by $\T^1$ and $\T^2$.
Thus, each of
 $\T^1$ and $\T^2$ is a \emph{rooted} tree with root at $x_0$ and, if
$\ell\in (0,1)$ is the length of the pendant of $\T^1$,
then   the pendant of $\T^2$ has a length of $1-\ell$ (see Figure \ref{figxz}.(b)).

Denoting by $u^i$ the restriction of $u$ to $\T^i$,
by a surgery procedure we will now build a new function $v\in H^1(\T)$, as follows.
First, let $J$ be the compact graph made up of one edge of length 1, with two
pendants of length $\ell$ attached to one endpoint and two other pendants of
length $1-\ell$ attached to the other endpoint (see Figure \ref{figxz}.(c)).
Then \emph{duplicate}  $\T^1$, and attach
a copy of it to each of the two pendants of $J$ having length $1-\ell$;
 similarly,
duplicate  $\T^2$, and attach
a copy of it to each of the two pendants of $J$ having length $\ell$
(see Figure \ref{figsurg}).
Finally, for $i\in\{1,2\}$, define $v(x)=u^i(x)$ on each copy of $\T^i$
and
let $v\equiv M$ on $J$.

In this way, one obtains a tree isometric to $\T$ with a function $v\in H^1(\T)$,
such that $v\equiv M$ on a set of total length 3, while elsewhere $v$ essentially
consists of two copies of $u$, suitably joined together.
It is clear that
\[
\int_\T|v'|^2\dx=2\int_\T|u'|^2\dx\,,\qquad\int_\T|v|^2\dx=2\int_\T|u|^2\dx+3M^2
\]
and therefore, writing down the  inequality analogue to \eqref{poinc1} satisfied by $v$ and
dividing by 2, one obtains the reinforced
Poincar\'e inequality \eqref{poincarest} for $u$, with $C=3\lambda_1/2$.
Finally, if $x_0$ is a vertex of $\T$, removing $x_0$ would disconnect $\T$ into
\emph{three} connected components, but the previous argument (formally unchanged) still works,
  with one
connected components acting as $\T^1$ (with $\ell=1$) and the other two (joined together)
acting as $\T^2$.

The case where $\G=\To$ is easily reduced to the previous one.
      For,  given $u\in H^1(\To)$, we can join three copies of $\To$ at their roots,
      thus obtaining $\T$ with a threefold version of $u$ on it: writing down \eqref{poincarest}
      on $\T$, and then dividing by 3, yields the corresponding inequality for $u$ on $\To$,
      with $C/3$ in place of $C$.
\endproof

\newcommand\ppcc{0.7}
\newcommand\umppcc{0.3}
\newcommand\disegnaTuno[2]{
node [nodino] (rootT1) {} child [rotate=#1,grow cyclic,shape=circle,level distance=\ppcc*#2]
    { node [nodo]
       {} child [level distance=#2,sibling angle=100] foreach \B in {1,2}
        { { node [nodo] {} child [dashed,level distance=0.5*#2,sibling angle=80]  foreach \C in {1,2} }
        }
    }
}
\newcommand\disegnaTdue[2]{
node  [nodino] {} child [rotate=#1,grow cyclic,shape=circle,level distance=\umppcc)*#2]
    { node [nodo]
       {} child [level distance=#2,sibling angle=100] foreach \B in {1,2}
        { { node [nodo] {} child
        [dashed,level distance=0.5*#2,sibling angle=80]  foreach \C in {1,2} }
        }
    }
}

\newcommand\sidel{3.4em}
\newcommand\angolo{50}
\newcommand\dista{0.5*\sidel}
\begin{figure}
\begin{center}
\begin{tikzpicture}[thick]
\path \disegnaTuno{90}{\sidel} \disegnaTdue{-90}{\sidel} node [right] {$\,x_0$};
\node at (-1.5*\sidel,2.25*\sidel) {(a)};
\end{tikzpicture}\qquad
\begin{tikzpicture}[thick]
\path +(0,0.25*\sidel)
\disegnaTuno{90}{\sidel} node [right] {$\,x_0$} +(0,-0.25*\sidel) \disegnaTdue{-90}{\sidel} node [right] {$\,x_0$};
\node at (1,0.875*\sidel) {$\T^1$};
\node at (1,-0.6*\sidel) {$\T^2$};
\begin{scope}[xshift=-0.15*\sidel]
\draw [<->,very thin] (0,0.25*\sidel) -- +(0,\sidel*\ppcc);
\node at ($  (0,0.25*\sidel) +(0,0.5*\sidel*\ppcc)$) [left] {\scriptsize $\ell\,\,$};
\draw [<->,very thin] (0,-0.25*\sidel) -- +(0,-\sidel*\umppcc);
\node at ($ (0,-0.25*\sidel) +(0,-0.5*\sidel*\umppcc) $) [left] {\scriptsize $1-\ell\,\,$};
\end{scope}
\node at (-1.5*\sidel,2*\sidel) {(b)};
%
%
\begin{scope}[xshift=3.1*\sidel]
\draw (0,\sidel) node  [nodo] (n1) {} -- ++(0,-\sidel)  node [nodo] (n2) {}
-- +(-90-\angolo:\ppcc*\sidel) node [nodino] (c) {};
\draw (n1) -- +(90+\angolo:\sidel*\umppcc) node [nodino] (a) {};
\draw (n1) -- +(90-\angolo:\sidel*\umppcc) node [nodino] (b) {};
\draw (n2) -- +(-90+\angolo:\ppcc*\sidel) node [nodino] (d) {};
\draw [<->,very thin] ($(n2) +(\angolo:0.15*\sidel)$)  -- ($(d) +(\angolo:0.15*\sidel)$);
\node at ( $(n2)!0.5!(d) +(\angolo:0.1*\sidel)$) [above right] {\scriptsize $\!\ell$};
\draw [<->,very thin] ($(n1) +(-\angolo:0.15*\sidel)$)  -- ($(b) +(-\angolo:0.15*\sidel)$);
\node at ($(n1)!0.5!(b) +(-\angolo:0.1*\sidel)$) [below right] {\scriptsize $1-\ell$};

\node at (-1,0.5*\sidel) {$J$};
\node at (-1,2*\sidel) {(c)};
\end{scope}
\end{tikzpicture}
\end{center}
\caption{(a) The point $x_0$ on $\T$. (b) The two connected components $\T^1$ and $\T^2$.
(c) The junction $J$.}\label{figxz}
\end{figure}
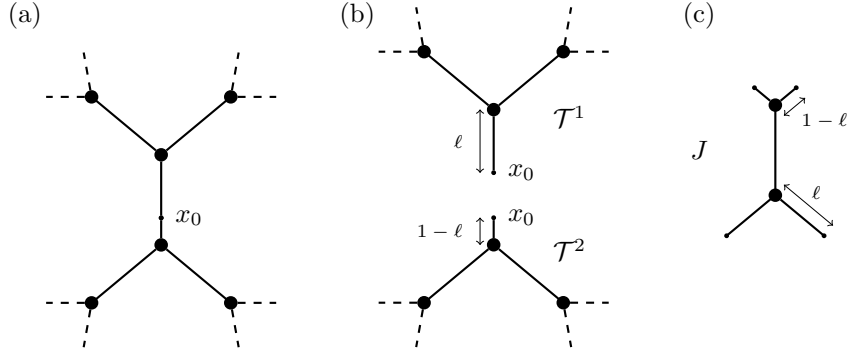

\begin{figure}
\begin{center}
\begin{tikzpicture}[thick]
\draw (0,0.5*\sidel) node  [nodo] (n1) {} -- (0,-0.5*\sidel)  node [nodo] (n2) {}
-- +(-90-\angolo:\ppcc*\sidel) node [nodino] (c) {};
\draw (n1) -- +(90+\angolo:\sidel*\umppcc) node [nodino] (a) {};
\draw (n1) -- +(90-\angolo:\sidel*\umppcc) node [nodino] (b) {};
\draw (n2) -- +(-90+\angolo:\ppcc*\sidel) node [nodino] (d) {};
\coordinate (a2) at ($(a) +(90+\angolo:\dista)$);
\coordinate (b2) at ($(b) +(90-\angolo:\dista)$);
\coordinate (c2) at ($(c) +(-90-\angolo:\dista)$);
\coordinate (d2) at ($(d) +(-90+\angolo:\dista)$);
\path (a2) \disegnaTuno{90+\angolo}{\sidel} ;
\path (b2) \disegnaTuno{90-\angolo}{\sidel} ;
\path  (c2)  \disegnaTdue{-90-\angolo}{\sidel};
\path (d2)  \disegnaTdue{-90+\angolo}{\sidel};
\node at (0,0) [rotate=0,right]  {$J$};
\path (n1) +(90+\angolo:2*\sidel) node {$\T^1$};
\path (n1) +(90-\angolo:2*\sidel) node {$\T^1$};
\path (n2) +(-90+\angolo:2*\sidel) node {$\T^2$};
\path (n2) +(-90-\angolo:2*\sidel) node {$\T^2$};
\path (n2) +(1.3*\sidel,0.1*\sidel) node (j1) {\scriptsize join};
\path (n1) +(1.1*\sidel,-0.1*\sidel) node (j2) {\scriptsize join};
\path (n2) +(-1.1*\sidel,0.1*\sidel) node (j1l) {\scriptsize join};
\path (n1) +(-1.3*\sidel,-0.1*\sidel) node (j2l) {\scriptsize join};
\draw [->,very thin,shorten >=0.2em] (j1) -- (d);
\draw [->,very thin,shorten >=0.2em] (j1) -- (d2);
\draw [->,very thin,shorten >=0.2em] (j2) -- (b);
\draw [->,very thin,shorten >=0.2em] (j2) -- (b2) ;
\draw [->,very thin,shorten >=0.2em] (j1l) -- (c2);
\draw [->,very thin,shorten >=0.2em] (j1l) -- (c);
\draw [->,very thin,shorten >=0.2em] (j2l) -- (a2);
\draw [->,very thin,shorten >=0.2em] (j2l) -- (a);
\end{tikzpicture}
\end{center}
\caption{The joint $J$, attached to two copies of $\T^1$ and two copies of $\T^2$.}\label{figsurg}
\end{figure}
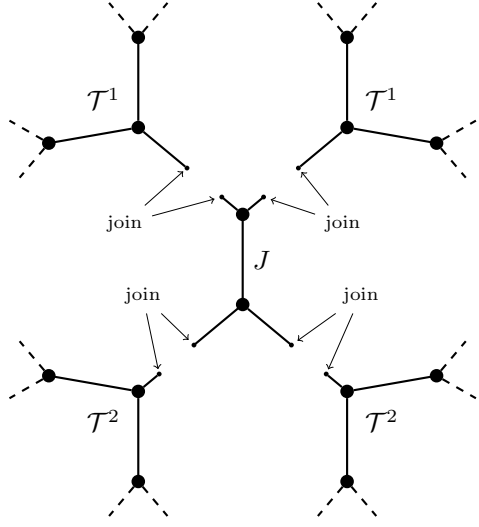

\begin{remark}
  As a byproduct, from \eqref{poincarest} one obtains the inequality
  \begin{equation}
    \label{stimalinf2}
    \Vert u\Vert_\infty\leq C \Vert u'\Vert_2\quad
    \forall u\in H^1(\G),
  \end{equation}
  for some universal constant $C>0$.
\end{remark}

A fundamental tool for the a priori estimates involving the NLS energy defined in \eqref{nlse}
is the Gagliardo-Nirenberg inequality
\begin{equation}
\label{gn}
\Vert u\Vert_p^p\leq C\,
\Vert u\Vert_2
^{\f{p}{2}+1}
\Vert u'\Vert_2
^{\f{p}{2}-1}\qquad\forall u\in H^1(\G),
\end{equation}
which is valid for $p\geq 2$ on any  metric graph $\G$ of infinite total length,
with a constant $C$ depending only on $p$ (see \cite{AST-JFA}).
Throught the paper, however, $\G$ will always denote either of the two trees
$\To$ and $\T$, and it will be tacitly understood that the exponent $p\in (2,6)$.
Recalling \eqref{defHumu}, \eqref{gn} can be rewritten as
\begin{equation}
\label{gnmu}
\Vert u\Vert_p^p\leq C\,
\mu^{\f{p+2}{4}}\,
\Vert u'\Vert_2
^{\f{p}{2}-1}\qquad\forall u\in H^1_\mu(\G).
\end{equation}

\begin{lemma}[a priori estimates] If $\G=\To$ or $\G=\T$,
then
\begin{align}
\label{coerc1}
\Vert u'\Vert_{L^2(\G)}^2 \leq 4 E(u,\G)+ C \mu^{\frac {p+2}{6-p}}
\quad\forall u\in H^1_\mu(\G),
\end{align}
for some $C=C(p)>0$. Moreover, if $\mu>0$
and
$u\in H^1_\mu(\G)$ is such that
\begin{equation}\label{usottoretta}
E(u,\G)\leq \frac 1 2\lambda_1\mu
\end{equation}
then
\begin{align}
\label{tesisottoretta}
1&\leq C   \Vert u\Vert_{L^\infty(\G)}^{p-4}\,\mu,\\
\label{tesisottoretta3}
\int_\G\frac{|u|^2}{(1+|x|)^2}\dx &\leq C \Vert u\Vert_{L^p(\G)}^p.
\end{align}
\end{lemma}
\begin{proof}
Since $p<6$, one can use the Young inequality to separate the product in
the right hand side of \eqref{gnmu},
and obtain
\[
\frac 1 p \Vert u\Vert_p^p\leq
\frac 1 4 \Vert u'\Vert_2^2+C
\mu^{\f{p+2}{6-p}}\,
\qquad\forall u\in H^1_\mu(\G),
\]
where $C=C(p)$.
Plugging into \eqref{nlse},  one obtains \eqref{coerc1}.

Furthermore, if $u\in H^1_\mu(\G)$, using \eqref{poincarest} we obtain
\begin{align*}
E(u,\G)&=\frac 1 2 \Vert u'\Vert_2^2 -\frac 1 p\Vert u\Vert_p^p\\
\geq &\frac {\lambda_1} 2 \Vert u\Vert_2^2 +C\Vert u\Vert_\infty^2-\frac 1 p\Vert u\Vert_p^p
=\frac 1 2\lambda_1\mu
+C\Vert u\Vert_\infty^2 -\frac 1 p\Vert u\Vert_p^p
\end{align*}
which, combined with \eqref{usottoretta}, gives
\begin{equation}
\label{tesisottoretta2}
\Vert u\Vert_\infty^2 \leq C \Vert u\Vert_p^p.
\end{equation}
Similarly, using \eqref{rest_rad} instead of \eqref{poincarest}, one obtains
\eqref{tesisottoretta3}. Finally, \eqref{tesisottoretta} is obtained
combining \eqref{tesisottoretta2}
with the inequality
\[
\Vert u\Vert_p^p
\leq \Vert u\Vert_\infty^{p-2} \Vert u\Vert_2^2
=\Vert u\Vert_\infty^{p-2} \,\mu,
\]
and rearranging terms (note that $\Vert u\Vert_\infty>0$ since $\mu>0$).
\end{proof}

Next we highlight some properties of the level function \eqref{elevel} that
will be widely used in the rest of the paper.

If $\mu\geq 0$ and $v\in H^1(\G)$ is not identically zero, renormalizing
$v$ we find
\begin{equation}
  \label{piega}
  \elevel_\G(\mu)\leq E\left(\sqrt \mu\,\frac{v}{\Vert v\Vert_2},\G\right)
  =\frac 1 2 \frac \mu{\Vert v\Vert_2^2}\Vert v'\Vert_2^2
  -
  \frac 1 p \left(\frac {\mu}{\Vert v\Vert_2^2}\right)^{\frac p 2}\Vert v\Vert_p^p.
\end{equation}
If, moreover, $\mu > \Vert v\Vert_2^2$, since $p>2$ we obtain
\begin{equation}
  \label{piega2}
  \elevel_\G(\mu)<\frac \mu{\Vert v\Vert_2^2}\left(
  \frac 1 2 \Vert v'\Vert_2^2
  -
  \frac 1 p \Vert v\Vert_p^p\right)=\frac \mu{\Vert v\Vert_2^2} E(v,\G).
\end{equation}
Finally, since by scaling $H^1_\mu(\G)=\{ \sqrt \mu\, v: v\in H^1_1(\G)\}$,
\eqref{elevel} can be rewritten as
\begin{equation}\label{infconc}
  \elevel_\G(\mu)=\inf_{v\in H^1_1(\G)}E\left(\sqrt\mu\,v,\G\right)=
  \inf_{v\in H^1_1(\G)}\left\{
  \frac 1 2 \mu\Vert v'\Vert_2^2
  -
\frac {\mu^{\frac p 2}}p\Vert v\Vert_p^p\right\},
\end{equation}
having used in the last passage the identity that stems from \eqref{piega}
when $\Vert v\Vert_2=1$.

We are now in a position to prove Theorem~\ref{THM0}.

\proof[{\bf Proof of Theorem \ref{THM0}}] First,  $\elevel_\G(0) =0$ by definition,
while for $\mu>0$ the infimum in \eqref{elevel} is always finite due to \eqref{coerc1}.
Next, since for every fixed $v\in H^1_1(\G)$ the quantity within braces in \eqref{infconc}
is a concave function of $\mu$, we see that
$\elevel_\G(\mu)$, being the lower envelope of concave functions, is concave
(this is typical of constrained
functionals having the form \eqref{nlse}, see e.g. \cite{AST-JFA}, Theorem 3.1).
Then, neglecting the last term in \eqref{piega}
one has
\[
\elevel_\G(\mu)
\leq \frac \mu 2
\,
\frac {\Vert v'\Vert_2^2}{\Vert v\Vert_2^2}
\qquad\forall v\in H^1(\G)\quad \text{($v\not\equiv 0$),}
\]
and taking the infimum over $v$ one obtains \eqref{sottoretta} from \eqref{rayleigh}.
Now, when $u\in H^1_\mu(\G)$,
$\mu\leq C \Vert u'\Vert_2^2$ by \eqref{poinc1}:
since moreover $\frac p 2 -1 <2$,
as a byproduct of \eqref{gnmu}
we have
\begin{equation*}
\Vert u\Vert_p^p\leq C\,
\mu^{\frac{p}2-1}\| u'\|_2^2.
\end{equation*}
Hence, recalling \eqref{nlse}, using the last inequality   we obtain
\[
E(u,\G)\geq
\frac {1} 2 \Vert u'\Vert_2^2 -
\frac C {\,p\,} \mu^{\frac p 2-1}\Vert u'\Vert_2^2
\geq
\frac {1} 2 \Vert u'\Vert_2^2\left(1-C\mu^{\frac p 2-1}\right).
\]
If $\mu$ is small enough, the last factor is positive, and using \eqref{poinc1}
again we find
\[
E(u,\G)\geq
\frac {1} 2 \lambda_1 \Vert u\Vert_2^2\left(1-C\mu^{\frac p 2-1}\right)
=\frac {1} 2 \lambda_1 \mu\left(1-C\mu^{\frac p 2-1}\right)\quad\forall u\in H^1_\mu(\G),
\]
so that
 $\elevel_\G(\mu) \ge \frac {1}2 \lambda_1 \mu(1 - C\mu^{\frac p 2 -1})$ if $\mu$ is small
 enough. Therefore,
 \[
 \liminf_{\mu \downarrow 0} \frac{\elevel_\G(\mu)}\mu\geq \frac 1 2\lambda_1,
 \]
 while the reverse inequality (for the limsup) is clear from \eqref{sottoretta}.
 Hence \eqref{derivL} is established
and, by concavity, this also shows that $\elevel_\G(\mu)$ is strictly positive and increasing in a right neighborhood of zero.

Finally, choosing any $v\in H^1(\G)$ ($v\not\equiv 0$), we see from \eqref{piega}
that
$\elevel_\G(\mu)<0$
if $\mu$ is large enough.
Therefore $\elevel_\G(\mu)$ achieves a positive maximum at some $\bar \mu>0$,
after which it becomes decreasing.
\endproof

As concerns the relations between $\elevel_{\To}$ and $\elevel_\T$, the following
lemma will be used in the proof of Theorem~\ref{teocfrlev}.

\begin{lemma}\label{lemmacfr}
There holds
\begin{equation}
  \label{dislevel}
  \elevel_{\To}(\mu) \le \elevel_\T(\mu)\quad\forall\mu\geq 0.
\end{equation}
On the other hand, if there exists a ground state of mass $\mu>0$ in $H^1_\mu(\T)$,
 then
\begin{equation}
\label{cfr}
  \elevel_{\To}(\mu) < \elevel_\T(\mu).
\end{equation}

\end{lemma}
\begin{proof}
The functions $u\in\Hmu(\T)$ with compact support
in $\T$ are dense (in the $H^1$ norm) in $\Hmu(\T)$, and every such $u$
can be regarded as an element of $\Hmu(\To)$ by identifying its
support with a subset of $\To$ and setting $u\equiv 0$ elsewhere on $\To$.
This density argument proves \eqref{dislevel}.

Now let $u\in H^1_\mu(\T)$ be a ground state of mass $\mu$.
  Split $\T$ at a vertex, thus creating three rooted trees $\T_1, \T_2,\T_3$, and
  call $u_i$ the restriction of $u$ to $\T_i$.
  Since $|u|>0$,
letting $\mu_i =\|u_i\|_{L^2(\T_i)}^2$ we have $\mu_i\in (0,\mu)$.
Then  \eqref{piega2} can be applied with $\G=\T_i$ (which is isometric to $\To$) and $v=u_i$,
 yielding
\[
\frac{\mu_i}\mu \elevel_{\To}(\mu)=
\frac{\mu_i}\mu \elevel_{\T_i}(\mu)< E(u_i,\T_i),\quad  i\in\{1,2,3\}.
\]
Thus, summing over $i$, we obtain
\[
\elevel_{\To}(\mu)<\sum_{i=1}^3 E(u_i,\T_i)
=E(u,\T).
\]
Since  $u$ is a ground state,  $E(u,\T)=\elevel_\T(\mu)$
and \eqref{cfr} is established.
\end{proof}

\section{Existence and non-existence of ground states}
\label{sec:existence}

In this section we study the compactness properties of minimizing sequences, deriving conditions that ensure the existence of ground states.

\begin{lemma}[Dichotomy]
\label{dichotomy}
Let $\G=\T$ or $\G =\To$. Given $\mu>0$,  let $\{u_n\}\subset \Hmu(\G)$ be a minimizing sequence for $E$, i.e.
\[
\lim_{n\to\infty} E(u_n,\G)=\elevel_\G(\mu),
\]
and assume that $u_n\rightharpoonup u$ weakly in $H^1(\G)$.
Then either
\begin{itemize}
\item[(i)] $u\equiv 0$, or
\item[(ii)] $u\in \Hmu(\G)$ and $E(u,\G) = \elevel_\G(\mu)$, i.e. $u$ is a ground state
of mass $\mu$.
\end{itemize}
\end{lemma}
The proof is similar to that of Lemma 4.1 in \cite{ADST}.
Here, for completeness, we present a simplified version of it.
\begin{proof}
Let $m = \|u\|_2^2$, so that $m \in [0,\mu]$. First, further assuming that
$m\in (0,\mu)$,
we shall find a contradiction.
Up to subsequences, we may assume that $u_n(x) \to u(x)$ a.e. on $\G$. Then, according to
the Brezis--Lieb Lemma (\cite{BL}), we can write
\begin{equation}
\label{s20}
E(u_n,\G) = E(u_n -u,\G) + E(u,\G) + o(1)\quad\text{as $n\to\infty$,}
\end{equation}
and, since $u_n\rightharpoonup u$ in $L^2(\G)$, as $n\to\infty$ we also have
\begin{equation} \label{massloss}
\Vert u_n-u\Vert_2^2=
\Vert u_n\Vert_2^2+
\Vert u\Vert_2^2
-2 \langle u_n,u\rangle_2
\to \mu- m.
\end{equation}
Therefore, since by assumption $m\in (0,\mu)$,
if  $n$ is large enough then $ u_n \not\equiv u$ and $\| u_n - u \|_2^2 < \mu$,
and \eqref{piega2} applied with $v=u_n-u$ gives
\[
E(u_n - u,\G) > \frac{\| u_n - u \|_2^2}\mu\, \elevel_\G(\mu).
\]
Hence, taking the liminf in \eqref{s20},
from the last inequality and \eqref{massloss}
we find
\[
\elevel_\G(\mu) \ge \frac{\mu -m}\mu\, \elevel_\G(\mu) +E(u,\G),
\]
i.e. $\elevel_\G(\mu) \ge \frac\mu m E(u,\G)$. But this is impossible since
it violates \eqref{piega2}, which holds true when $v=u$ since
$u\not\equiv 0$ and $\mu>m$.

This shows that either $m = 0$ (i.e. $u\equiv 0$) or $m = \mu$.
In the latter case, $u_n\to u$ strongly in $L^2(\G)$ by \eqref{massloss}, hence also in $L^p(\G)$
since the $u_n$ are uniformly bounded.
Therefore $u \in H_\mu^1(\G)$ and, by weak lower semicontinuity,
\[
E(u,\G) \le \liminf_n \frac12 \|u_n'\|_2^2 -\lim_n \frac1p\|u_n\|_p^p = \liminf_n E(u_n,\G) = \elevel_\G(\mu),
\]
which shows that $u$ is a minimizer.			
\end{proof}

The previous lemma leads to the following effective criterion for
the existence of ground states.

\begin{proposition}
\label{conv}
Let $\G=\T$ or $\G =\To$. Given $\mu>0$,  let $\{u_n\}\subset \Hmu(\G)$ be a minimizing sequence for $E$, i.e.
$\lim_{n} E(u_n,\G)=\elevel_\G(\mu)$.
If
\begin{equation}
\label{infty}
\liminf_n \|u_n\|_{L^\infty(\G)} >0,
\end{equation}
then  there exists a  ground state $u\in H^1_\mu(\G)$ of mass $\mu$.
\end{proposition}

\begin{proof} We split the proof into two parts.

\medskip
		
\noindent{\em Part 1: $\G=\T$.} Since the isometry group of $\T$ acts transitively on its
edges, by applying a rigid motion to each $u_n$
we can further assume that each $u_n$ achieves its $L^\infty$ norm on
a given edge $e\subset\T$  independent of $n$.
Due to \eqref{coerc1}, $\{u_n\}$ is bounded in the $H^1$ norm, hence
for a subsequence (not relabeled) $u_n\rightharpoonup u$ for some $u\in H^1(\T)$ and,
in particular, $u_n\to u$ uniformly on $e$.
Thus, if $u\equiv 0$, then
 $u_n \to 0$ also in $L^\infty(\G)$
because all the $L^\infty$ norms are achieved on $e$.
But this is incompatible with \eqref{infty},
hence $u\not\equiv 0$ and,  by Lemma \ref{dichotomy}, $u$ is a ground state of mass $\mu$.
\medskip
			
\noindent{\em Part 2: $\G=\To$.} As before,
assuming  that $u_n\rightharpoonup 0$ in $H^1(\G)$, we shall find a contradiction, but
the argument  is more involved  since $\To$
has no translation invariance.
Identifying the edge containing the root with the interval $[0,1]$
($x=0$ corresponding to the root), we have $u_n \to 0$ uniformly on $[0,1]$,
by $L^\infty_{\rm loc}$ convergence.
Set $\eps_n = u_n(0)$ (so that $\eps_n\to 0$) and define $v_n :\To \to \R$ as
\[
v_n(x) = \begin{cases} u_n(x) - (1-x)\eps_n & \text{ if } x\in [0,1] \\
u_n(x) & \text{ elsewhere on } \To. \end{cases}
\]
Clearly, $v_n(0) = 0$ for each $n$, $v_n -u_n \to 0$ strongly in $H^1(\To)$ and $\|v_n\|_\infty = \|u_n\|_\infty$
for every $n$ large.
Finally, define $w_n:\To \to \R$ by
\[
w_n(x) = \frac{\sqrt\mu}{\|v_n\|_2}\,v_n(x)
\]
so that $w_n \in \Hmu(\To)$ and, plainly,
\begin{equation}
\label{cof}
E(w_n,\To) = E(u_n,\To) +o(1)
\end{equation}
as $n \to \infty$, while $\liminf_n \|w_n\|_\infty >0$ still holds. Identifying $\To$ with a subtree of $\T$, and recalling that $w_n(0)=0$, we can view the functions $w_n$ as elements of $\Hmu(\T)$,
after extending them to zero on $\T\setminus \To$. By \eqref{cof} and \eqref{dislevel},
\[
\elevel_\T(\mu)\leq\lim_n E(w_n,\T)=\lim_n E(u_n,\To)=\elevel_{\To}(\mu)
\leq \elevel_{\T}(\mu)
\]
which reveals
that $w_n$ is a minimizing sequence for $E(\,\cdot\, , \T)$ in $H^1_\mu(\T)$.
Since $w_n$ does not tend to zero in $L^\infty(\T)$, by
Part~1 there exists a ground state
of mass $\mu$ in $\Hmu(\T)$,
but this is a contradiction, since  the last chain of inequalities is
incompatible with \eqref{cfr}.
\end{proof}
		
Now, recalling \eqref{sottoretta}, we can show that a ground state
of mass $\mu$ exists, as soon as the function $\elevel_\G(\mu)$ detaches
from the linear function $\frac 1 2\lambda_1\mu$.

\begin{proposition}
\label{levelargument}
Let $\G=\T$ or $\G =\To$.
If $\elevel_\G(\mu)<\frac12\eign\mu$ or, equivalently,
if $\mu>\musG$,
 then there exists
a ground state $u\in \Hmu(\G)$ of mass $\mu$.
\end{proposition}		

\begin{proof} Since $\elevel_\G(\mu)$ is concave and $\elevel_\G(0)=0$,
the equivalence of the two assumptions is immediate from \eqref{sottoretta} and \eqref{critmass}.

Now let $\{u_n\} \subset \Hmu(\G)$ be a minimizing sequence for $E$ on $\Hmu(\G)$.
If \eqref{infty} were false then
a subsequence (still denoted by $u_n$) would  satisfy
$\Vert u_n\Vert_\infty\to 0$, hence
$u_n\to 0$ also in $L^p$
since $\Vert u_n\Vert_2^2=\mu$ and $p>2$.
In this case, by the Poincar\'e inequality \eqref{poinc1} we would have
\[
\elevel_\G(\mu)=\lim_n E(u_n,\G)=\lim_n \frac12\|u_n'\|_2^2 -\lim_n \frac1p \|u_n\|_p^p\geq\f12\eign\mu,
\]
which contradicts our assumption. Hence \eqref{infty}
is necessarily satisfied, and Proposition \ref{conv} applies.
\end{proof}

As a useful criterion for the existence of ground states we state the following simple consequence of the preceding result.

\begin{corollary}
Let $\G=\T$ or $\G =\To$. The level $\elevel_\G(\mu)$ is achieved
by a ground state
if, and only if, there exists a function $u\in \Hmu(\G)$ such that $E(u,\G) \le \frac12 \lambda_1 \mu$.
\end{corollary}

\begin{proof} Let $u\in \Hmu(\G)$ satisfy $E(u,\G) \le \frac12 \lambda_1 \mu$.
If $u$ does not achieve $\elevel_\G(\mu)$, then $\elevel_\G(\mu)<\frac12 \lambda_1\mu$,
and a ground state exists
by Proposition \ref{levelargument}. Conversely, if $u \in \Hmu(\G)$ is
a ground state, then $E(u,\G) = \elevel_\G(\mu) \le \frac12 \lambda_1 \mu$ by \eqref{sottoretta}.
\end{proof}

\begin{proposition}
\label{prop_non existence}
Let $\G=\T$ or $\G= \To$. If $\musG>0$ and $\nu\in (0,\musG)$,
then there is no ground state of mass $\nu$.
\end{proposition}

\begin{proof}
  Let $v\in H^1_\nu(\G)$ be an arbitrary function of mass  $\nu\in (0,\musG)$.
  Then \eqref{piega2}, applied with $\mu=\musG$, yields
\[
\frac 1 2\lambda_1\musG =  \elevel_\G(\musG)< \frac {\musG}\nu E(v,\G),
  \]
so that $E(v,\G)>\frac 1 2\lambda_1\nu$.
Since by \eqref{sottoretta} $\frac 1 2\lambda_1\nu\geq \elevel_\G(\nu)$,
$v$ cannot be a ground state.
\end{proof}

\begin{theorem}\label{teocfrlev}
There hold $\musTo\leq \musT$ and
\begin{equation}
  \label{cfrlev}
  \elevel_{\To}(\mu)\,\,\,
  \begin{cases}
  \,\, = \,\,\elevel_\T(\mu)\,\, =\,\,\frac 1 2\lambda_1\mu &\text{if $\mu\in [0,\musTo]$},\\{}\\
  \,\,< \,\,\elevel_\T(\mu) &\text{if $\mu>\musTo$.}
  \end{cases}
\end{equation}
Moreover, if $p\in (4,6)$, then $0<\musTo<\musT$.
\end{theorem}
\begin{proof}
The inequality $\musTo\leq \musT$
  follows from \eqref{critmass} and \eqref{dislevel},  hence in particular
$  \elevel_{\To}(\mu)
 =\elevel_\T(\mu)=\frac 1 2\lambda_1\mu$ for every $\mu\in [0,\musTo]$.
Now consider any $\mu>\musTo$, so that $\elevel_{\To}(\mu)<\frac 1 2\lambda_1\mu$ by \eqref{critmass}
and \eqref{sottoretta}. If  $\elevel_{\T}(\mu)=\frac 1 2\lambda_1\mu$, then
the strict inequality in \eqref{cfrlev} is obvious; if, on the other hand,
$\elevel_{\T}(\mu)<\frac 1 2\lambda_1\mu$,
then
a ground state
  $u\in H^1_\mu(\T)$ exists by Proposition~\ref{levelargument}, and
the strict inequality in  \eqref{cfrlev} follows from Lemma~\ref{lemmacfr}.

Finally, when $p\in (4,6)$,
  $\musTo>0$ by Proposition~\ref{critm>0}, and there exists a ground state of mass
  $\musT$ in $H^1_{\musT}(\T)$ by Lemma~\ref{gsmcrit}. Therefore,
   Lemma~\ref{lemmacfr} guarantees the strict inequality $\elevel_{\To}(\musT)<
   \elevel_\T(\musT)=\frac 1 2\lambda_1\musT$ and, consequently, $\musT>\musTo$.
\end{proof}

\begin{proposition}
\label{critm>0}
Let $\G=\T$ or $\G=\To$. If $p\in[4,6)$, then $\musG>0$.
\end{proposition}
	
\begin{proof}
Take an arbitrary mass $\mu>\musG$. Since, by \eqref{critmass} and \eqref{sottoretta},
$\elevel_\G(\mu)<\frac 1 2\lambda_1\mu$, there exists
$u\in\Hmu(\G)$ satisfying \eqref{usottoretta}. If $p=4$, then \eqref{tesisottoretta}
reduces to $C\mu\geq 1$, and letting $\mu\downarrow \musG$ yields $C\musG\geq 1$.
When $p>4$, \eqref{tesisottoretta} is even more restrictive:
combining \eqref{stimalinf2} with \eqref{coerc1} we have
\[
\Vert u\Vert_\infty^2\leq C_1 E(u,\G)+C_2 \mu^{\frac {p+2}{6-p}}
\leq
\frac {C_1} 2\lambda_1\mu+C_2\mu^{\frac {p+2}{6-p}},
\]
so that $\Vert u\Vert_\infty^2\leq C_3\mu$
if $\mu$ is small enough. Then, again,
\eqref{tesisottoretta} yields an a priori lower bound $\mu\geq C_4>0$,
whence $\musG\geq C_4$.
\end{proof}

\begin{lemma}\label{gsmcrit}
 Let $\G=\To$ or $\G=\T$. If $p\in (4,6)$, then there exists a ground state of mass $\musG$.
\end{lemma}
\begin{proof}
Set $\mu_n=\musG+\frac 1 n$, let
$v_n\in H^1_{\mu_n}(\G)$ be a ground state of mass $\mu_n$ (which exists by
Proposition~\ref{levelargument}),  and define the renormalized functions
\begin{equation}\label{defvnun}
u_n:=\sqrt{\frac{\musG}{\mu_n}}\,\,v_n\in H^1_{\musG}(\G).
\end{equation}
Since the $v_n$'s are equibonded in $H^1(\G)$ by \eqref{coerc1},
and since $\musG/\mu_n \to 1$, as $n\to\infty$  we have
\begin{align*}
E(u_n,\G) &=\frac12\frac{\musG}{\mu_n}\int_\G|v_n'|^2\dx-
\frac1p\Big(\frac{\musG}{\mu_n}\Big)^{\frac{p}2}\int_\G|v_n|^p\dx
 = E(v_n,\G) + o(1)\\ &=\elevel_\G(\mu_n)+o(1) \le \frac12 \lambda_1 \mu_n + o(1) = \frac12 \lambda_1 \musG + o(1)
\end{align*}
having used \eqref{sottoretta}.
Recalling that $\elevel_\G(\musG)=\frac 1 2\lambda_1\musG$, this shows that
$\{u_n\}$ is a minimizing sequence for $E$ on $H^1_{\musG}(\G)$. Furthermore,
since each $v_n$ satisfies $E(v_n,\G) \le \frac12 \lambda_1\mu_n$,
we can let $u=v_n$ and $\mu=\mu_n$ in \eqref{tesisottoretta}, i.e.
\[
1\leq C\Vert v_n\Vert_\infty^{p-4}\,\mu_n\quad\forall n,
\]
with $C$ independent of $n$. Finally, since $\mu_n\to \musG>0$,
this estimate is inherited by the $u_n$'s via \eqref{defvnun},
and we obtain that $\Vert u_n\Vert_\infty\geq\delta$
for some $\delta>0$ independent of $n$.
Then, the existence of a ground state of mass $\musG$ follows
from Proposition~\ref{conv}.
\end{proof}

\proof[{\bf Proof of Theorem \ref{THM 1}}] Note that \eqref{inf} is immediate
from \eqref{sottoretta} and \eqref{critmass}.
The part concerning ground states has been proved in the previous lemmas and propositions.
\endproof

\section{Radial ground states: proof of Theorem \ref{THM 2}}
\label{sec:rad}
First we observe that the radial problem \eqref{GSPrad},	when
$\G=\To$ and for any given $\mu>0$, is equivalent
to the same problem when $\G=\T$, as soon as $\mu$ is replaced with
$3\mu$. Indeed,  regarding $\T$ as
three copies of $\To$ joined together at a given vertex $\radice\in\T$, any radial function
$u\in H^1_{r,3\mu}(\T)$ consists of three copies of the same
radial function in $\Hr(\To)$ and, clearly, the converse is also true.
Therefore, recalling \eqref{radlevel} and \eqref{critmassrad}, we have
\[
\elevel_{\To,r}(\mu)=
\frac 1 3\elevel_{\T,r}(3\mu)\quad\forall\mu\geq 0,\qquad
\musTor=\frac 1 3\musTr,
\]
and there is a one-to-one correspondence between radial ground states in $H^1_\mu(\To)$
and radial ground states in $H^1_{3\mu}(\T)$.
Thus, in proving Theorem~\ref{THM 2}, it suffices to consider
the case where $\G=\To$.

Thanks to the following proposition, several results  from the previous sections
become available also for the radial problem.

\begin{proposition}
  The constant $\lambda_1$ is unaltered, if  the infimum of the Rayleigh
  quotient in \eqref{rayleigh} is restricted to radial functions. In other words,
\begin{equation}
\label{rayleighsym}
\eign=\inf_{\substack{u\in H^1(\To)\\\text{$u$
radial, $\not\equiv 0$}}}\frac{\int_{\To}|u'|^2\dx}{\int_{\To}|u|^2\dx}.
\end{equation}
\end{proposition}
The claim is an indirect consequence  of very general results  concerning
the Laplacian on metric trees (see \cite{solomyak}, in particular \S~3.2).
For completeness, however, we give a self-contained proof, based on a quadratic average
of $u(x)$ along points at the same depth.
Given $u\in H^1(\To)$, let
\[
v(t)=\sqrt{\frac 1 {\# X(t)}\sum_{x\in X(t)}
u(x)^2},\quad t\geq 0,\quad X(t)=\left\{x\in\To\,:\,\,|x|=t\right\}
\]
(note that $\# X(t)=2^{\lfloor t\rfloor}$ for a.e. $t$) and define
a \emph{radial} function $w\in H^1(\To)$ by letting $w(x):=v(|x|)$.
After elementary computations, using Cauchy--Schwarz one can easily check that
\[
\Vert w\Vert_2^2=
\Vert u\Vert_2^2,\quad
\Vert w'\Vert_2^2\leq
\Vert u'\Vert_2^2,
\]
so that passing from $u$ to $w$
does not increase the Rayleigh quotient. This symmetrization procedure,
however, \emph{decreases} the $L^p$ norm, hence
the NLS energy might increase
(by this technique, therefore, one cannot solve the symmetry   question for ground states).

\begin{remark}\label{remlevsym} The concavity of the radial level function
$\elevel_{\To,r}$ defined in \eqref{radlevel}, its negativity for large $\mu$, and
the inequality
\begin{equation}
\label{sottorettarad}
\elevel_{\To,r} (\mu) \le \frac12 \lambda_1\mu\qquad \forall \mu \ge 0,
\end{equation}	
can be proved
exactly as done for $\elevel_\G$  in Section~\ref{sec:prelim}, in the
proof of Theorem~\ref{THM0}. Indeed, it suffices to replace everywhere
the word
``function'' with ``radial function'', $\elevel_\G$ with
$\elevel_{\To,r}$,
$H^1_\mu(\G)$ with
$H^1_{\mu,r}$ (as defined in \eqref{defHunomur}) and so on, and refer to
\eqref{rayleighsym} instead of \eqref{rayleigh}. Moreover, with the same notational
changes,
also Lemma~\ref{dichotomy} remains valid (now writing ``\emph{radial} ground state'')
in the radial setting.
\end{remark}

Contrary to Lemma~\ref{dichotomy},
the proof of Proposition~\ref{conv} cannot be adapted to the radial case,
since in Part~1 each function $u_n$ (that now would be radial, with respect
to some fixed origin $\radice\in\T$) is initially
subjected to a rigid motion, hence the function $u$
obtained in the limit might fail to be radial. Therefore, a different
proof will be given,
based on the fact that
any radial function in $H^1(\To)$ has an exponential decay away from the root.
\begin{lemma}\label{lemmadecay}
If $u\in H^1(\To)$ is radial, then
\begin{equation}
\label{espdecay}
  |u(x)|^2 \leq 2^{-|x|\,} C\, \Vert u'\Vert_{L^2(\To)}^2\qquad\forall x\in \To,
\end{equation}
for some universal constant $C$ independent of $u$.
\end{lemma}
\begin{proof}
Given an integer $d\geq 0$,
let $\G_i$ ($1\leq i\leq 2^d$) denote the $2^d$ subtrees of $\To$
at depth $d$, that is,
the $2^d$ rooted trees
having, as  root, a vertex $y\in \To$ such that $|y|=d$
(e.g., if $d=3$, in Fig. \ref{Figalberi}~(a) these trees are those starting with a dashed line).
Applying  \eqref{stimalinf2} with
$\G$ replaced by one of these $\G_i$'s yields
$\Vert u\Vert_{L^\infty(\G_i)}^2\leq C \Vert u'\Vert_{L^2(\G_i)}^2$
but, since by radial symmetry the restrictions of $u$ to each $\G_i$ are all equal,
summing over $i$ we obtain
\[
2^d \, \Vert u\Vert_{L^\infty(\G_i)}^2\leq
C\sum_{i=1}^{2^d}\int_{\G_i} |u'|^2\,dx
=
C\int_{\bigcup \G_i} |u'|^2\,dx
\leq C\int_{\To} |u'|^2\,dx,
\]
which can be rewritten as
\[
\Vert u\Vert_{L^\infty(\G_i)}^2\leq
2^{-d} \,C\,
\Vert u'\Vert_{L^2(\To)}^2,\quad
1\leq i\leq 2^d.
\]
But every $x\in \To$ belongs to
one of the subtrees $\G_i$ at depth $d$, as soon as $d\leq |x|$. Choosing $d$ as  the integer
part of $|x|$, one has $2^{-d}\leq 2^{1-|x|}$,
and \eqref{espdecay} follows immediately.
\end{proof}

We can now prove the radial analogue of Proposition~\ref{conv}.

\begin{proposition}
\label{convrad}
Given $\mu>0$,  let $\{u_n\}\subset H^1_{\mu,r}(\To)$ be a minimizing sequence for
the radial problem \eqref{GSPrad},
i.e.
$\lim_{n} E(u_n,\To)=\elevel_{\To,r}(\mu)$.
If
\begin{equation}
\label{inftyrad}
\liminf_n \|u_n\|_{L^\infty(\To)} >0,
\end{equation}
then  there exists a radial ground state $u\in H^1_{\mu,r}(\To)$ of mass $\mu$.
\end{proposition}
\begin{proof}
  By \eqref{coerc1}, $\{u_n\}$ is bounded in $H^1(\To)$ and therefore (up to subsequences)
  $u_n\rightharpoonup u$ in $H^1(\To)$ (whence $u_n\to u$ in $L_{\text{loc}}^\infty(\To)$)
   for some radial $u \in H^1(\To)$.
 Now observe that, since by \eqref{coerc1} $\Vert u'_n\Vert_2^2\leq C$, from
 \eqref{espdecay} we see that $|u_n(x)|^2 \leq C 2^{-|x|}$ with $C$ independent of $n$,
 hence the convergence $u_n\rightharpoonup u$ in $L_{\text{loc}}^\infty(\To)$
 is, in fact, a convergence in $L^\infty(\To)$. Then \eqref{inftyrad} implies
 that $u\not=0$, hence $u$ is a radial ground state by the radial version of
 Lemma~\ref{dichotomy} (see Remark~\ref{remlevsym}).
\end{proof}

For radial functions with relatively small energy, a consequence of
the exponential decay \eqref{espdecay}
is the following lower bound for the mass and the $L^\infty$ norm.
\begin{lemma}\label{lemmalowerbs}
Let $p\in (2,6)$. If $\mu>0$ and $u\in H^1_{\mu,r}(\To)$ is any radial function
  such that $E(u,\To)\leq \frac 1 2\lambda_1\mu$, then
\begin{equation}
  \label{lowerbs}
  \mu\geq \delta_1,\qquad \Vert u\Vert_{L^\infty(\To)}^2 \,\mu^{\frac{p-2}{6-p}}\geq\delta_2
\end{equation}
for some two constants $\delta_i>0$  depending only on $p$.
\end{lemma}
\begin{proof}
Under our assumptions, we can combine \eqref{tesisottoretta3}
with
the estimate
\[
\int_{\To}|u|^p\dx
\leq\left(\sup_{x\in\To}|u(x)|^{p-2}(1+|x|)^2 \right)\int_{\To}\f{|u|^2}{(1+|x|)^2}\dx,
\]
to obtain
\begin{equation}\label{dis2}
1\leq C \sup_{x\in\To}|u(x)|^{p-2}(1+|x|)^2,
\end{equation}
whence in particular
\begin{equation}\label{dis3}
1\leq C \Vert u\Vert_\infty^{\frac {p-2}2}\cdot\sup_{x\in\To}|u(x)|^{\frac{p-2}2}(1+|x|)^2.
\end{equation}
Now, since $u$ is radial, using \eqref{espdecay} we obtain, for every $r>0$,
\[
\sup_{x\in\To}|u(x)|^{r}(1+|x|)^2
\leq
C_r \,\Vert u'\Vert_2^r
\]
where $C_r$ depends only on $r$. Moreover, using \eqref{coerc1}
combined with $E(u,\G)\leq \frac 1 2\lambda_1\mu$,
we obtain from the previous inequality
\begin{equation}
\label{dis4}
\sup_{x\in\To}|u(x)|^{r}(1+|x|)^2
\leq
C_r \left( 2\lambda_1\mu+C \mu^{\frac {p+2}{6-p}}\right)^{\frac {r}2}\quad\forall r>0.
\end{equation}
Choosing $r=p-2$ and plugging into \eqref{dis2}, one proves the first bound in
\eqref{lowerbs}.
In the light of this, since $\frac{p+2}{6-p}>1$, \eqref{dis4} can be simplified
to
\begin{equation*}
\sup_{x\in\To}|u(x)|^{r}(1+|x|)^2
\leq
\tilde C_r \mu^{\frac {p+2}{6-p} \,\cdot\,\frac {r}2}\quad\forall r>0.
\end{equation*}
Choosing $r=\frac{p-2}2$ and plugging into \eqref{dis3}, one proves the second bound in
\eqref{lowerbs}.
\end{proof}

This enables us to show that the radial analogue of Proposition~\ref{critm>0} holds true (even in
a stronger form, being valid regardless of $p$).

\begin{proposition}
\label{radmu>0}
For every $p\in (2,6)$,  one has $\mu_{\To,r}^*>0$.
\end{proposition}
	
\begin{proof} Consider any mass $\mu>\musTor$. Then, by \eqref{critmassrad} and
\eqref{sottorettarad},
$\elevel_{\To,r}(\mu)<\frac 1 2\lambda_1\mu$, so that there exists
a radial function $u\in H^1_{\mu,r}(\To)$ satisfying
the assumptions of Lemma~\ref{lemmalowerbs}. Therefore
$\mu\geq \delta_1>0$,
and the claim follows letting $\mu\downarrow \musTor$.
\end{proof}

\begin{proposition}\label{propsym}
Let $p\in (2,6)$. If $\mu\geq \musTor$, then there exists a radial ground state $u\in H^1_{\mu,r}(\To)$
of mass $\mu$. Conversely, if $\nu\in (0,\musTor)$, there is no radial ground state
of mass $\nu$.
\end{proposition}
\begin{proof}
If $\mu>\musTor$, the existence of a radial ground state follows by arguing exactly
as in the proof of Proposition~\ref{levelargument}, now relying on Proposition~\ref{convrad}
instead of \ref{conv}, with the notational changes discussed in Remark~\ref{remlevsym}.
Similarly, for the last part of the statement one can proceed as in the proof
of Proposition~\ref{prop_non existence}, choosing as $v$ a radial function.

Finally, we prove the existence of a radial ground state of mass $\musTor$, but even though
the argument is similar to that used for Lemma~\ref{gsmcrit},
some additional estimates are needed
because here there is no restriction on $p$.
Thus, set $\mu_n=\musTor+\frac 1 n$ and let  $v_n$
denote a \emph{radial} ground state of mass $\mu_n$ (which exists by the first
part of this proof).
Since $E(v_n,\To)=\elevel_{\To,r}\leq\frac 1 2\lambda_1 \mu_n$ by \eqref{sottorettarad},
Lemma~\ref{lemmalowerbs} applies and the second bound in \eqref{lowerbs}
reads
\begin{equation}
\label{dis6}
\Vert v_n\Vert_\infty^2 \,\mu_n^{\frac{p-2}{6-p}}\geq\delta_2>0.
\end{equation}
Now,
arguing
as in the proof of Lemma~\ref{gsmcrit},
one shows that the \emph{radial} functions
\begin{equation}\label{defvnun2}
u_n:=\sqrt{\frac{\musTor}{\mu_n}}\,\,v_n\in H^1_{\musTor}(\To)
\end{equation}
are a minimizing sequence for the radial ground state problem with
mass $\musTor$.
Since $\mu_n\to \musTor>0$, we find from \eqref{defvnun2} and \eqref{dis6}
\[
\liminf_n
\Vert u_n\Vert_\infty
=
\liminf_n
\Vert v_n\Vert_\infty
>0,
\]
hence a radial
ground state of mass $\musTor$ exists by Proposition~\ref{convrad}.
\end{proof}

\proof[{\bf Proof of Theorem \ref{THM 2}}]
As observed at the beginning of this section, it suffices to consider
the case where $\G=\To$. The part concerning the level function	
$\elevel_{\To,r}$ follows from Remark~\ref{remlevsym}, whereas the other
claims have been proved in Proposition~\ref{radmu>0} and Proposition~\ref{propsym}.
\endproof

\end{document}